\documentclass[11pt]{article}
\usepackage[T1]{fontenc}
\usepackage{lmodern}
\usepackage[margin=1in]{geometry}
\usepackage{amsmath,amssymb,amsthm,mathtools}
\usepackage{microtype}
\usepackage{enumitem}
\usepackage[hidelinks]{hyperref}
\usepackage{tikz}
\usetikzlibrary{arrows.meta}

\newtheorem{theorem}{Theorem}[section]
\newtheorem{lemma}[theorem]{Lemma}
\newtheorem{proposition}[theorem]{Proposition}
\newtheorem{corollary}[theorem]{Corollary}
\newtheorem{conjecture}[theorem]{Conjecture}
\newtheorem{question}[theorem]{Question}
\theoremstyle{definition}
\newtheorem{definition}[theorem]{Definition}

\newcommand{\doi}[1]{\href{https://doi.org/#1}{doi:\nolinkurl{#1}}}
\newcommand{\arxiv}[1]{\href{https://arxiv.org/abs/#1}{arXiv:\nolinkurl{#1}}}
\newcommand{\cP}{\mathcal P}
\newcommand{\cQ}{\mathcal Q}

\newcommand{\grotzschdrawing}[1]{%
\begin{tikzpicture}[
    x=0.88cm,y=0.88cm,
    >=Stealth,
    vertex/.style={
        circle,draw,fill=white,
        minimum size=6.5mm,
        inner sep=0pt,
        font=\small
    },
    edge/.style={#1,line width=0.8pt}
]
\foreach \i/\a in {1/90,3/18,5/-54,2/-126,4/162}
    \node[vertex] (y\i) at (\a:3) {\(y_{\i}\)};

\foreach \i/\a in {1/90,4/18,2/-54,5/-126,3/162}
    \node[vertex] (x\i) at (\a:1.55) {\(x_{\i}\)};

\node[vertex] (z) at (0,0) {\(z\)};

\foreach \a/\b in {3/1,4/2,5/3,4/1,5/2}
    \draw[edge] (y\a)--(y\b);

\foreach \a/\b in {3/1,4/1,2/2,3/2,1/3,2/3,1/4,5/4,4/5,5/5}{%
    \draw[edge] (y\a)--(x\b);
}

\foreach \i in {1,...,5}
    \draw[edge] (z)--(x\i);
\end{tikzpicture}%
}

\numberwithin{equation}{section}

\title{Clique Number, Colorings, and Substitution in Tournaments}
\author{Guantao Chen \and Shengze Wang}
\date{}
\hypersetup{
    pdftitle={Clique Number, Colorings, and Substitution in Tournaments},
    pdfauthor={Guantao Chen and Shengze Wang}
}

\begin{document}
\maketitle

\begin{abstract}
The clique number of a tournament is the minimum clique number of its
backedge graph over all vertex orders. For each \(k\geq3\), we construct
regular, vertex-transitive Cayley tournaments that are simultaneously
\(k\)-critical for the clique and dichromatic numbers. Their orders
include all \(N>1\) with \(N\equiv1\pmod{m_k}\) for a positive even
integer \(m_k\). A reduction from tournament
two-colorability that adds four vertices proves that deciding whether
the clique number is at most two is NP-complete. We also construct poset
tournaments of clique number two and unbounded dichromatic number. Their
clique-optimal backedge graphs can have arbitrarily large girth.
The minimum order of a tournament with no vertex order simultaneously
attaining its clique and dichromatic numbers is eleven. Our example has
a unique clique-optimal order, whose backedge graph is the four-chromatic
Gr\"otzsch graph. Finally, we give bounds for the clique number under
substitution, exhibit two substitutions with the same outer tournament
and block clique numbers but different clique numbers, and obtain an
exact formula and a criticality criterion for cyclic substitutions
whose blocks satisfy \(\omega=\chi\).
\end{abstract}



\section{Introduction}
\label{sec:introduction}

All graphs and tournaments in this paper are finite, and all graphs are
simple. We take tournaments to be nonempty unless otherwise stated and set
\(\omega(\varnothing)=\chi(\varnothing)=0\).
A tournament is an orientation of a complete graph. We write
\(x\to y\) for an arc directed from \(x\) to \(y\), and \(A\to B\)
when all arcs between disjoint sets \(A,B\) point from \(A\) to
\(B\). A tournament is \emph{transitive} if it has a vertex order
in which every arc points forward.

For a tournament \(T\) and an order
\(\sigma=(v_1,\ldots,v_n)\), the \emph{backedge graph} \(T^\sigma\)
has vertex set \(V(T)\) and edge set
\(E(T^\sigma)=\{v_iv_j:i<j\text{ and }v_j\to v_i\}\).
We write \(x\prec_\sigma y\) when \(x\) precedes \(y\), and
\(\sigma|_A\) for the restricted order on \(A\).
Following Kim~\cite{Kim2013} and Aboulker, Aubian, Charbit, and
Lopes~\cite{AboulkerEtAl2026}, define
\(\omega(T)=\min_\sigma\omega(T^\sigma)\), where clique and
chromatic numbers of graphs have their usual meanings.

The \emph{dichromatic number} \(\chi(T)\), introduced for digraphs
by Neumann-Lara~\cite{NeumannLara1982}, is the minimum number of
transitive sets partitioning \(V(T)\). A \(t\)-coloring means a
partition into at most \(t\) such sets; empty classes are allowed.
Equivalently, \(\chi(T)=\min_\sigma\chi(T^\sigma)\).
Indeed, concatenating transitive orders of the color classes makes each
class independent in the resulting backedge graph. Conversely, every
independent set of \(T^\sigma\) is transitive in the restricted order. In particular,
\(\omega(T)\leq\chi(T)\), and \(\omega(T)=1\) precisely when \(T\)
is transitive.

We write \(K_1\) for the one-vertex tournament and \(C_3\) for
the cyclic triangle. A subtournament always means an induced subtournament.
Restricting orders and colorings shows that both parameters are
monotone under taking subtournaments. Writing
\(T-v=T[V(T)\setminus\{v\}]\), we also have
\(\theta(T)-1\leq\theta(T-v)\leq\theta(T)\) for
each \(\theta\in\{\omega,\chi\}\). Appending \(v\) to an optimal
order increases the backedge clique number by at most one, and giving
\(v\) its own color extends any coloring of \(T-v\).
A tournament is \emph{\(\theta\)-critical} if
\(\theta(T-v)<\theta(T)\) for every vertex \(v\).
It is \emph{\(k\)-\(\theta\)-critical} if \(\theta(T)=k\) and
\(\theta(T-v)=k-1\) for every \(v\).
Thus every vertex-critical tournament of parameter \(k\) is
\(k\)-critical. For either parameter, the unique one-critical
tournament is a single vertex, and the unique two-critical
tournament is the cyclic triangle.

Let \(T\) be a tournament and let
\(\mathcal S=(S_v)_{v\in V(T)}\) be a family of pairwise
vertex-disjoint nonempty tournaments.
The \emph{substitution} \((T,\mathcal S)\) replaces each vertex
\(v\) by \(S_v\), with all arcs directed from \(S_u\) to \(S_v\)
whenever \(u\to v\) in \(T\).
We call \(T\) the \emph{outer tournament} and the tournaments
\(S_v\) the \emph{blocks}.

Repeated blocks always mean disjoint copies.

For pairwise disjoint nonempty tournaments \(A,B,C\), let
\(\Delta(A,B,C)\) be the tournament with these three blocks and
all interblock arcs oriented as \(A\to B\to C\to A\).

Neumann-Lara and Urrutia constructed infinite families
of regular \(k\)-\(\chi\)-critical tournaments for every \(k\geq3\)
except \(k=4\), and Neumann-Lara settled the remaining
case~\cite{neumann-urrutia,neumann1997}.
Neumann-Lara also constructed infinite families of
\(k\)-\(\chi\)-critical circulant tournaments for every \(k\geq3\)
except \(k=7\); Araujo-Pardo and Olsen completed that
case~\cite{neumann2000,araujo-olsen}.
The problem here is to attain criticality for both parameters in the
same tournament.

Aboulker et al. asked the following question.

\begin{question}[{\cite[Question~49]{AboulkerEtAl2026}}]
\label{question:bounded-certificate}
Is there a function \(\ell:\mathbb N\to\mathbb N\) such that,
for every integer \(k\geq1\), every tournament \(T\) with
\(\omega(T)\geq k\) contains a subtournament \(A\) satisfying
\(|V(A)|\leq\ell(k)\) and \(\omega(A)\geq k\)?
\end{question}

A positive answer is equivalent to having only finitely
many pairwise nonisomorphic \(k\)-\(\omega\)-critical tournaments
for every fixed \(k\). The same authors proposed the following
conjecture.

\begin{conjecture}[{\cite[Conjecture~50]{AboulkerEtAl2026}}]
\label{conjecture:critical}
For each integer \(k\geq3\), infinitely many pairwise
nonisomorphic tournaments are \(k\)-\(\omega\)-critical.
\end{conjecture}

Crew, Fan, Koerts, Moore, and Spirkl~\cite{CrewEtAl2026} proved a
threshold version of the bounded-certificate statement, with the
stronger hypothesis \(\omega(T)\geq h(k)\) for a suitable function
\(h\). Aubian and Coulomb~\cite{AubianCoulomb2026} proved the
criticality conjecture for \(k=3,4\), and
Lelarge~\cite{lelarge-k5} gave a construction for \(k=5\).
Our first theorem proves Conjecture~\ref{conjecture:critical}
and answers Question~\ref{question:bounded-certificate} negatively.

Let \(N\) be odd. A set
\(D\subseteq\mathbb Z_N\setminus\{0\}\) is
\emph{maximal antipodal-free} if it contains exactly one element
of every pair \(\{d,-d\}\) with \(d\ne0\).
The \emph{Cayley tournament}
\(\operatorname{Cay}(\mathbb Z_N,D)\), also called a circulant
tournament, has vertex set \(\mathbb Z_N\), with \(i\to j\)
precisely when \(i-j\in D\). Translations are automorphisms, so it
is vertex-transitive; every vertex has outdegree \((N-1)/2\),
so it is regular.

\begin{theorem}
\label{thm:critical}
For every integer \(k\geq3\), there is a positive even integer
\(m_k\) such that every integer \(N>1\) with
\(N\equiv1\pmod{m_k}\) is the order of a regular, vertex-transitive
Cayley tournament that is simultaneously \(k\)-\(\omega\)-critical
and \(k\)-\(\chi\)-critical. Consequently, there are infinitely
many pairwise nonisomorphic \(k\)-\(\omega\)-critical tournaments
for every \(k\geq3\).
\end{theorem}

Every proper subtournament of a \(k\)-\(\omega\)-critical
tournament has clique number at most \(k-1\).
The unbounded orders in Theorem~\ref{thm:critical} therefore
rule out a bounded certificate at each fixed \(k\geq3\) under the
hypothesis \(\omega(T)\geq k\).

Aubian~\cite{Aubian2024} proved that deciding \(\omega(T)\leq k\)
is NP-complete for every fixed \(k\geq3\).
Aubian and Coulomb~\cite[Theorem~4.6]{AubianCoulomb2026} give a polynomial-time
algorithm that, on inputs satisfying \(\omega(T)\leq2\), finds an
order with backedge clique number at most \(100\), and leave
exact recognition at threshold two open. The remaining problem
can be stated as follows.

\begin{question}[{\cite[Section~6]{AboulkerEtAl2026}; \cite{AubianCoulomb2026}}]
\label{question:recognition}
What is the computational complexity of deciding, for an input
tournament \(T\), whether \(\omega(T)\leq2\)?
\end{question}

\begin{theorem}
\label{thm:recognition}
Deciding whether \(\omega(T)\leq2\) is NP-complete.
\end{theorem}

The reduction adds exactly four vertices to an instance of tournament
two-colorability. This completes the fixed-threshold classification.
Threshold one is the transitivity test, and recognition is NP-complete
at every fixed threshold \(k\geq2\).
Corollary~\ref{cor:digraph-recognition} extends NP-completeness
to general digraphs at every fixed threshold at least two.
Corollary~\ref{cor:promise-order} rules out a polynomial-time
algorithm that always finds a triangle-free backedge order on
inputs satisfying \(\omega(T)\leq2\), unless
\(\mathrm P=\mathrm{NP}\).

An order \(\alpha\) of a tournament \(T\) is a \emph{poset ordering}
if, whenever \(x\prec_\alpha y\prec_\alpha z\), \(x\to y\), and
\(y\to z\), we have \(x\to z\).
A tournament admitting a poset ordering is a \emph{poset tournament}.

A class \(\mathcal C\) of tournaments is \emph{\(\chi\)-bounded}
if there is a function \(f:\mathbb N\to\mathbb N\) such that
\(\chi(T)\leq f(\omega(T))\) for every \(T\in\mathcal C\).

Aboulker et al. proposed the following conjecture.

\begin{conjecture}[{\cite[Conjecture~14]{AboulkerEtAlPOD2026}}]
\label{conjecture:poset}
Poset tournaments form a \(\chi\)-bounded class.
\end{conjecture}

The following family disproves the conjecture already at clique number two.

\begin{theorem}
\label{thm:poset}
For every integer \(m\geq1\), there is an explicitly constructed
poset tournament \(T_m\) with
\(\omega(T_m)=2\) and \(\chi(T_m)\geq m+1\).
Consequently, the class of poset tournaments is not \(\chi\)-bounded.
\end{theorem}

An \emph{\(\omega\)-order} attains
\(\min_\sigma\omega(T^\sigma)\), and a \emph{\(\chi\)-order}
attains \(\min_\sigma\chi(T^\sigma)\). An order is
\emph{simultaneously optimal} if it is both.
Aboulker et al. asked whether such an order always exists.

\begin{question}[{\cite[Question~5]{AboulkerEtAl2026}}]
\label{question:simultaneous}
Does every tournament \(T\) admit an order \(\sigma\) satisfying both
\(\omega(T^\sigma)=\omega(T)\) and
\(\chi(T^\sigma)=\chi(T)\)?
\end{question}

Define
\(\chi_\omega(T)=
\min\{\chi(T^\sigma):\omega(T^\sigma)=\omega(T)\}\).
Then such an order exists exactly when
\(\chi_\omega(T)=\chi(T)\).
If \(\omega(T)=\chi(T)\), every \(\chi\)-order is simultaneously
optimal. Since clique number one characterizes transitivity,
the smallest possible parameter pair admitting an obstruction is
\((\omega(T),\chi(T))=(2,3)\).

The next theorem answers Question~\ref{question:simultaneous}
negatively and determines the smallest counterexample.

\begin{theorem}
\label{thm:smallest}
Every tournament on at most ten vertices has a simultaneously
optimal order. There is an eleven-vertex tournament \(T\) with
\(\omega(T)=2\), \(\chi(T)=3\), and \(\chi_\omega(T)=4\), whose
unique \(\omega\)-order \(\sigma_0\) satisfies
\(T^{\sigma_0}\cong M(C_5)\).
\end{theorem}

Here \(M(C_5)\) is the Gr\"otzsch graph, obtained by applying
Mycielski's construction to a five-cycle.

Section~\ref{sec:templates} proves the critical Cayley theorem.
Sections~\ref{sec:recognition}--\ref{sec:smallest} give the reduction,
the poset construction, and the minimum obstruction to simultaneous
optimality. Section~\ref{sec:substitution} studies substitution. It
includes two substitutions with the same outer tournament and the same
block clique numbers whose clique numbers differ. The Cayley construction
and the cyclic substitution criterion appeared in an earlier version
of this work~\cite{ChenWangCayley2026}.

\section{Critical Cayley tournaments}
\label{sec:templates}

Write \(N^-(v)=\{u:u\to v\}\). The construction embeds a tournament
of clique number \(k\) in every in-neighborhood of a larger tournament.
The first vertex in any order then extends a backedge clique of size
\(k\) to one of size \(k+1\). Rational coordinates give the matching
upper bounds and allow the construction to be iterated.

\begin{definition}[\(k\)-template]
\label{def:template}
Let \(k\geq2\) be an integer. A \emph{\(k\)-template} is a pair
\((T,\varphi)\), where \(T\) is a tournament with \(\omega(T)=k\)
and \(\varphi:V(T)\to[0,k-1]\cap\mathbb Q\) is an injective map
satisfying the following conditions:
\begin{enumerate}[label=(T\arabic*)]
\item if \(\varphi(u)<\varphi(v)\) and \(v\to u\), then
\(\varphi(v)-\varphi(u)\geq1\); and
\item there are vertices \(a_0,\ldots,a_{k-1}\), called
\emph{anchors}, such that \(\varphi(a_i)=i\) for
\(0\leq i\leq k-1\) and \(a_j\to a_i\) whenever
\(0\leq i<j\leq k-1\).
\end{enumerate}
\end{definition}

Since \(\varphi\) is injective, it induces a linear order
\(\sigma_\varphi\) on \(V(T)\). Every backedge with respect to
this order has span at least \(1\), and the anchors form a
backedge clique of size \(k\). For vertices \(u,v\in V(T)\)
with \(\varphi(u)<\varphi(v)\), we call \((u,v)\) a
\emph{backward pair} if \(v\to u\), and a \emph{forward pair}
if \(u\to v\). In either case, the \emph{span} of \((u,v)\)
is \(\varphi(v)-\varphi(u)\).

\begin{lemma}
\label{lem:template-lifting}
Let \((T,\varphi)\) be a \(k\)-template, where \(k\geq2\).
There is a positive integer \(M\) such that, for every even
integer \(s\geq2\), setting \(N=kMs+1\) gives a Cayley tournament
\(T_D=\operatorname{Cay}(\mathbb Z_N,D)\) that is simultaneously
\((k+1)\)-\(\omega\)-critical and
\((k+1)\)-\(\chi\)-critical.
Moreover, \(\varphi'(r)=r/(Ms)\), for \(0\leq r\leq N-1\),
makes \((T_D,\varphi')\) a \((k+1)\)-template.
\end{lemma}

\begin{proof}
Let \(\cP(T,\varphi)\) consist of all values \(1+\varphi(v)\),
for \(v\in V(T)\), together with the spans of all backward
pairs. Let \(\cQ(T,\varphi)\) be the set of spans of all
forward pairs.

The template is \emph{separated} if
\(\cP(T,\varphi)\cap\cQ(T,\varphi)=\varnothing\) and every backward
pair that is not an anchor--anchor pair has span strictly greater
than \(1\). When the template is clear, we write simply
\(\cP\) and \(\cQ\).

We may first replace the coordinate map by one that makes the
template separated. To see this, list the nonanchors as
\(b_1,\ldots,b_m\) so that
\(\varphi(b_1)<\cdots<\varphi(b_m)\), and put
\(\delta=\min\{|\varphi(u)-\varphi(v)|:
u,v\in V(T),\ u\neq v\}\).
Since the anchors have consecutive integer coordinates, we have
\(0<\delta\leq1\). For a positive rational number \(t<\delta\),
define \(\varphi_t(a_i)=i\) and
\(\varphi_t(b_i)=\varphi(b_i)+t^{m-i+1}\) for \(1\leq i\leq m\).

If \(i<j\), then \(t^{m-i+1}<t^{m-j+1}\), and hence
\(\varphi_t(b_j)-\varphi_t(b_i)
=\varphi(b_j)-\varphi(b_i)+t^{m-j+1}-t^{m-i+1}
>\varphi(b_j)-\varphi(b_i)\).
Thus every span between two nonanchors increases.
Also, every nonanchor moves to the right by at most \(t<\delta\).
It follows that \(\varphi_t\) has range in \([0,k-1]\), is
injective, and gives the same linear order as \(\varphi\).
The tournament \(T\) and the anchor coordinates are unchanged,
so \(\omega(T)=k\) and (T2) still holds.

For (T1), every backward span increases or remains unchanged,
except possibly for a pair \((u,a_j)\) with \(u\) a nonanchor.
In this case, (T1) and injectivity imply that
\(\varphi(u)<\varphi(a_{j-1})\).
Since \(u\) moves by less than \(\delta\), we still have
\(\varphi_t(u)<\varphi_t(a_{j-1})\), and hence
\(\varphi_t(a_j)-\varphi_t(u)>1\).

It remains to separate \(\cP(T,\varphi_t)\) from \(\cQ(T,\varphi_t)\).
Each unwanted equality is a polynomial equation in \(t\). In a span
polynomial, each nonanchor endpoint contributes a distinct positive
power of \(t\), with negative sign at the left endpoint and positive
sign at the right. A forward span is nonconstant, since every
anchor--anchor pair is backward. If a backward span agreed identically
with a forward span, the coefficients would determine the same
nonanchor endpoints in the same positions. In the case of a single
nonanchor, the constant term would also determine the anchor. Thus
the two spans would come from the same ordered pair, an impossibility.

Nor can a forward span equal \(1+\varphi_t(v)\) identically. For an
anchor \(v\) the latter is constant. For a nonanchor \(v\), its
positive-degree term forces \(v\) to be the right endpoint and the
left endpoint to be an anchor. Comparing constants would place that
anchor at \(-1\). Each unwanted equality therefore excludes only
finitely many values of \(t\). Choose a rational \(t\in(0,\delta)\)
outside their union and use \(\varphi_t\) as the coordinate map from
now on. This gives a separated \(k\)-template.

Choose a positive integer \(M\) clearing the denominators of the
coordinates. Fix an even integer \(s\geq2\), put \(N=kMs+1\), and
identify residues in \(\mathbb Z_N\) with their representatives in
\(\{0,\ldots,N-1\}\). For each \(v\in V(T)\), set
\(x_v=(1+\varphi(v))Ms\).
The numbers \(x_v\) are distinct integers divisible by \(s\),
and \(Ms\leq x_v\leq kMs=N-1\).

Let \(\cP'\) consist of all coordinates \(x_v\) and all
differences \(x_v-x_u\) corresponding to backward pairs.
Let \(\cQ'\) consist of the differences \(x_v-x_u\)
corresponding to forward pairs.
Then \(\cP'=Ms\,\cP\) and \(\cQ'=Ms\,\cQ\), so
\(\cP'\cap\cQ'=\varnothing\).
The template conditions also give \(\min\cP'\geq Ms\) and
\(\max\cQ'\leq(k-1)Ms\).
Every element of \(\cP'\cup\cQ'\) is divisible by \(s\),
while \(N\equiv1\pmod{s}\).

We construct a maximal antipodal-free set \(D\) as follows.
First include \(\cP'\cup(N-\cQ')\) in \(D\), where
\(N-\cQ'=\{N-q:q\in\cQ'\}\).
From each remaining pair \(\{d,N-d\}\), put in \(D\) one
member that is at least \(Ms\).

This definition never puts both members of a complementary pair
in \(D\). Such a conflict among the prescribed elements would
give \(p_1+p_2=N\), \(q_1+q_2=N\), or \(p=q\), for some
\(p,p_1,p_2\in\cP'\) and \(q,q_1,q_2\in\cQ'\).
The first two equalities are impossible modulo \(s\), and the
last is excluded by \(\cP'\cap\cQ'=\varnothing\).
For each pair not already prescribed, a member that is at least
\(Ms\) exists because its two members sum to
\(N=kMs+1\geq2Ms+1\).

Thus \(D\) is a maximal antipodal-free set and defines the
tournament \(T_D=\operatorname{Cay}(\mathbb Z_N,D)\).
Translations act transitively on its vertices, and every vertex has
outdegree \(|D|=(N-1)/2\).

Every element of \(D\) is at least \(Ms\).
Indeed, the elements of \(\cP'\) are at least \(Ms\), and for
\(q\in\cQ'\) we have \(N-q\geq Ms+1\).
All elements chosen from the remaining pairs are at least
\(Ms\) by definition.

Let \(X=\{x_v:v\in V(T)\}\).
Since \(x_v\in\cP'\subseteq D=N^-(0)\) for every \(v\in V(T)\),
we have \(X\subseteq N^-(0)\).
Let \(u,v\in V(T)\) with \(\varphi(u)<\varphi(v)\).
If \(v\to u\) in \(T\), then \(x_v-x_u\in\cP'\subseteq D\),
so \(x_v\to x_u\) in \(T_D\).
If \(u\to v\) in \(T\), then \(x_v-x_u\in\cQ'\), so
\(N-(x_v-x_u)\in D\) and \(x_u\to x_v\) in \(T_D\).
Thus the map \(v\mapsto x_v\) is an isomorphism from \(T\)
to \(T_D[X]\).
For every \(w\in\mathbb Z_N\), translation by \(w\) shows that
\(w+X\) induces a copy of \(T\) contained in the in-neighborhood
\(N^-(w)\) of \(w\).

Consider the order \(0<1<\cdots<N-1\).
If \(i<j\) and \(j\to i\), then \(j-i\in D\), so \(j-i\geq Ms\).
Therefore, if \(y_1<\cdots<y_r\) is a backedge clique, then
\(y_r-y_1\geq(r-1)Ms\).
Since all vertices lie between \(0\) and \(kMs\), there is no
backedge clique of size \(k+2\), so \(\omega(T_D)\leq k+1\).
After deleting \(0\), the largest possible span is \(kMs-1\),
so the same argument gives \(\omega(T_D-0)\leq k\).
Translation gives \(\omega(T_D-v)\leq k\) for every
\(v\in V(T_D)\).

For the lower bounds, take any order of \(V(T_D)\) and let \(w\)
be its first vertex. The in-neighborhood \(N^-(w)\) contains the
copy of \(T\) induced by \(w+X\).
The order induced on this copy has a backedge clique \(K\) of
size \(k\). Every vertex of \(K\) points to \(w\), so
\(K\cup\{w\}\) is a backedge clique of size \(k+1\).
Thus \(\omega(T_D)\geq k+1\).
The same copy of \(T\) is a subtournament of \(T_D-w\), so
\(\omega(T_D-w)\geq k\).
Therefore \(T_D\) is \((k+1)\)-\(\omega\)-critical.

Return to the natural order \(0<1<\cdots<N-1\) for the dichromatic
upper bounds.
Each of the intervals
\(\{iMs,iMs+1,\ldots,(i+1)Ms-1\}\), for \(0\leq i\leq k-1\),
and the singleton \(\{kMs\}\) is transitive:
every positive difference inside one of the intervals is smaller
than \(Ms\), so it is not in \(D\).
Thus the earlier vertex points to the later one in this order.
Hence \(\chi(T_D)\leq k+1\).
Since \(\omega(T_D)=k+1\), equality holds.
After deleting \(0\), the \(k\) intervals
\(\{iMs+1,iMs+2,\ldots,(i+1)Ms\}\), for \(0\leq i\leq k-1\),
are transitive and partition the remaining vertices, so
\(\chi(T_D-0)\leq k\).
Since \(\omega(T_D-0)=k\) and \(\omega\leq\chi\), it follows
that \(\chi(T_D-0)=k\).
Translation gives \(\chi(T_D-v)=k\) for every \(v\in V(T_D)\).

Finally, define \(\varphi'(r)=r/(Ms)\) for \(0\leq r\leq N-1\).
This map is injective, rational-valued, and has range in \([0,k]\).
Every backedge in its order has span at least \(1\).
The vertices \(0,Ms,\ldots,kMs\) are anchors.
For \(0\leq i\leq k-1\), we have
\(x_{a_i}=(1+\varphi(a_i))Ms=(i+1)Ms\in\cP'\subseteq D\).
Thus \(hMs\in D\) for every \(1\leq h\leq k\), and so
\(jMs\to iMs\) whenever \(0\leq i<j\leq k\).
It follows that \((T_D,\varphi')\) is a \((k+1)\)-template,
as required.
\end{proof}

\begin{proof}[Proof of Theorem~\ref{thm:critical}]
The cyclic triangle provides a two-template.
Take vertices \(a_0,b,a_1\) with arcs
\(a_0\to b\to a_1\to a_0\), and assign coordinates
\(0,1/2,1\), respectively.
The only backedge in the coordinate order is \(a_1\to a_0\),
of span one, and \(a_0,a_1\) are anchors.

Starting with this template,
Lemma~\ref{lem:template-lifting} supplies a template at
every successive level.
Fix a \(k\)-template and its integer \(M\).
Put \(m_{k+1}=2kM\).
Every \(N>1\) with \(N\equiv1\pmod{m_{k+1}}\) can be written as
\(N=kMs+1\) for an even integer \(s\geq2\), so the lemma
gives simultaneous \((k+1)\)-criticality at every required order.
Each output is regular and vertex-transitive because it is a
Cayley tournament.
The permitted orders are unbounded, yielding infinitely many
pairwise nonisomorphic tournaments at every threshold at least
three.
\end{proof}
\section{Recognizing clique number two}
\label{sec:recognition}

For a tournament \(T\), form \(\widehat T\) by adding a vertex
\(a\) and a cyclic triangle on \(C=\{c_1,c_2,c_3\}\), with
\(c_1\to c_2\to c_3\to c_1\) and
\(\{a\}\to V(T)\to C\to\{a\}\).
Thus \(\widehat T\cong\Delta(K_1,T,C_3)\).
For a graph \(G\), write \(N_G(v)\) for the neighborhood of \(v\).

\begin{lemma}
\label{lem:four-vertices}
For every tournament \(T\),
\(\chi(T)\leq2\) if and only if \(\omega(\widehat T)\leq2\).
\end{lemma}

\begin{proof}
Suppose first that \(V(T)=A\mathbin{\dot\cup}B\), where both
\(T[A]\) and \(T[B]\) are transitive; either part may be empty.
Order each part transitively and take
\(\sigma=(A,c_1,a,B,c_2,c_3)\), with each part occupying the
indicated position in its chosen internal order.
Put \(G=(\widehat T)^\sigma\).
There are no edges of \(G\) inside \(A\) or \(B\).
Thus \(G[A\cup B]\) is bipartite. The neighborhoods of the four
added vertices are
\[
\begin{aligned}
N_G(a)&=A\cup\{c_2,c_3\}, & N_G(c_1)&=B\cup\{c_3\},\\
N_G(c_2)&=\{a\},          & N_G(c_3)&=\{a,c_1\}.
\end{aligned}
\]
Each is independent. Hence \(G\) is triangle-free, as required.

Conversely, let \(\sigma\) be an order of \(\widehat T\) with
triangle-free backedge graph \(G=(\widehat T)^\sigma\).
Some vertex of \(C\) precedes \(a\).
Otherwise, since every vertex of \(C\) dominates \(a\), the vertex
\(a\) is adjacent in \(G\) to all of \(C\).
Every order of the cyclic triangle \(\widehat T[C]\) has a
backedge, so \(G[C]\) contains an edge, giving a triangle together
with \(a\).

Choose \(c\in C\) with \(c\prec_\sigma a\), and put
\(V^-=\{u\in V(T):u\prec_\sigma a\}\) and
\(V^+=\{u\in V(T):a\prec_\sigma u\}\).
For every \(x\in V^-\), the arc \(a\to x\) is backward, so
\(V^-\subseteq N_G(a)\).
The neighborhood of a vertex in a triangle-free graph is
independent, and therefore \(G[V^-]\) is edgeless.
Thus \(T[V^-]\) is transitive in the restricted order.

Similarly, for every \(y\in V^+\), we have
\(c\prec_\sigma a\prec_\sigma y\) and \(y\to c\).
Hence \(V^+\subseteq N_G(c)\), so \(G[V^+]\) is edgeless and
\(T[V^+]\) is transitive.
The partition \(V^-,V^+\) is a two-coloring of \(T\), proving
the equivalence.
\end{proof}

\begin{proof}[Proof of Theorem~\ref{thm:recognition}]
A vertex order is a certificate for \(\omega(T)\leq2\), and its
backedge graph can be checked for triangles in \(O(n^3)\) time,
where \(n=|V(T)|\). Thus the problem belongs to NP.
Two-colorability of tournaments is NP-hard by Fox, Gishboliner,
Shapira, and Yuster~\cite[Theorem~5.1]{FoxGishbolinerShapiraYuster2019}.
The map \(T\mapsto\widehat T\) adds four vertices and constructs
the adjacency matrix in \(O((n+4)^2)\) time, without requiring a
coloring of \(T\).
Lemma~\ref{lem:four-vertices} therefore gives a polynomial-time
many-one reduction, proving NP-hardness.
\end{proof}

For a loopless digraph \(D\), possibly with opposite arcs,
define \(D^\sigma\) by placing an edge between two vertices
whenever an arc between them points backwards in \(\sigma\),
and put \(\omega(D)=\min_\sigma\omega(D^\sigma)\).

\begin{corollary}
\label{cor:digraph-recognition}
For every fixed integer \(k\geq2\), deciding whether a digraph
\(D\) satisfies \(\omega(D)\leq k\) is NP-complete, even when
the input is restricted to tournaments.
\end{corollary}

\begin{proof}
For fixed \(k\), a vertex order is a polynomially verifiable
certificate for \(\omega(D)\leq k\): check that \(D^\sigma\)
contains no \((k+1)\)-clique by inspecting all
\((k+1)\)-vertex subsets.
Thus the problem belongs to NP.
NP-hardness follows from the tournament restriction, by
Theorem~\ref{thm:recognition} for \(k=2\) and
Aubian's result~\cite{Aubian2024} for each fixed \(k\geq3\).
\end{proof}

\begin{corollary}
\label{cor:promise-order}
Unless \(\mathrm P=\mathrm{NP}\), there is no algorithm that
runs in polynomial time on every input tournament and returns
an order \(\sigma\) with \(\omega(T^\sigma)\leq2\) whenever
\(\omega(T)\leq2\).
\end{corollary}

\begin{proof}
Suppose such an algorithm exists.
Run it on an arbitrary tournament \(T\), and accept precisely
when its output is a vertex order whose backedge graph is
triangle-free.
If \(\omega(T)\leq2\), the guarantee ensures acceptance.
If \(\omega(T)>2\), no vertex order can pass the test.
The output and the absence of triangles can be checked in
polynomial time, giving a polynomial-time algorithm for the
NP-complete problem in Theorem~\ref{thm:recognition}.
\end{proof}

In particular, unless \(\mathrm P=\mathrm{NP}\), any
polynomial-time algorithm guaranteed to return an order with
\(\omega(T^\sigma)\leq c\) whenever \(\omega(T)\leq2\)
must have \(c\geq3\), where \(c\) is an integer.

\section{Poset tournaments}
\label{sec:poset}

Aubian and Coulomb~\cite[Section~5.2]{AubianCoulomb2026} used the
Blanche Descartes graphs as backedge graphs to construct tournaments
of clique number two and unbounded dichromatic number. To obtain a
poset order as well, we use the two linear extensions in the
cover-graph construction of Trotter and Wang~\cite{TrotterWang2014}.

\begin{proof}[Proof of Theorem~\ref{thm:poset}]
Fix integers \(m\geq1\) and \(g\geq4\). By Trotter and
Wang~\cite[Theorem~2.7]{TrotterWang2014}, there is a poset \(P\)
whose cover graph \(G\) has girth at least \(g\) and chromatic number
\(2m+1\). There are two linear extensions \(\alpha,\beta\) of \(P\)
such that no vertex lies strictly between the endpoints of an edge
of \(G\) in both orders. Since the endpoints of a cover relation
are comparable, they occur in the same relative order in
\(\alpha\) and \(\beta\).

Define \(T=T_{m,g}\) on \(V(G)\) by orienting each edge of \(G\)
backward in \(\beta\) and every other pair forward in \(\beta\).
Thus \(T^\beta=G\), and triangle-freeness gives \(\omega(T)\leq2\).

We show that \(\alpha\) is a poset ordering. Suppose
\(x\prec_\alpha y\prec_\alpha z\) and \(x\to y\to z\).
Every edge of \(G\) is backward in \(\alpha\), so neither \(xy\)
nor \(yz\) belongs to \(G\). Both arcs therefore point forward in
\(\beta\), giving \(x\prec_\beta y\prec_\beta z\).
If \(xz\) were an edge of \(G\), then \(y\) would lie between its
endpoints in both orders. Thus \(xz\) is a nonedge of \(G\) and
\(x\to z\), as required.

By the backedge coloring inequality of Nguyen, Scott, and
Seymour~\cite[4.1]{NguyenScottSeymour2025}
(see also~\cite[Proposition~5.8]{AubianCoulomb2026}),
\[
  2m+1=\chi(T^\beta)
    \leq \omega(T^\beta)\,\chi(T)
    \leq 2\chi(T).
\]
It follows that \(\chi(T)\geq m+1\geq2\). Hence \(T\) is
nontransitive and \(\omega(T)=2\).

For an explicit family \(T_m=T_{m,4}\), the recursion
in~\cite[Section~4]{TrotterWang2014} admits a simple choice of
auxiliary hypergraph. Start with their ordered five-cycle. At a stage
with \(n\) vertices and chromatic number \(r\), take all \(n\)-subsets
of an ordered set \(B\) of size \(r(n-1)+1\). For each such subset, match its vertices in increasing order to a
copy of the current graph in its first linear extension. Order the
copies lexicographically. In every \(r\)-coloring of
\(B\), some indexing subset is monochromatic. Its matched copy would
have to avoid that color, although it requires \(r\) colors. Conversely,
one new color suffices for \(B\). No triangle is introduced, since \(B\) is independent, distinct
copies are anticomplete, and each vertex of \(B\) has at most one
neighbor in each copy. The two linear extensions are
those of the same recursion. Iterating up to chromatic number \(2m+1\)
gives the asserted explicit examples. Since their dichromatic numbers
are unbounded at clique number two, the class is not \(\chi\)-bounded.
\end{proof}

For each fixed \(m\), the construction with a prescribed \(g\) gives
a tournament \(T_{m,g}\) and an \(\omega\)-order \(\beta\) for which
\(T_{m,g}^\beta\) has girth at least \(g\).

\section{The smallest counterexample}
\label{sec:smallest}

\begin{proof}[Proof of Theorem~\ref{thm:smallest}]

For a nonempty graph \(G\), its Mycielski graph \(M(G)\) is formed
by retaining \(G\), adding an independent copy of its vertex set,
joining each copied vertex to the original neighbors of its
counterpart, and adding a vertex adjacent to all copied vertices.
Mycielski~\cite{Mycielski1955} proved that
\(\chi(M(G))=\chi(G)+1\) and that \(M(G)\) is triangle-free
whenever \(G\) is triangle-free.
Thus \(M(C_5)\) is the eleven-vertex triangle-free four-chromatic
Gr\"otzsch graph.
Chv\'atal~\cite{Chvatal1974} proved that every triangle-free graph
on at most ten vertices is three-colorable.
We also use Neumann-Lara's theorem~\cite{NeumannLara1994} that
every tournament on at most ten vertices has dichromatic number
at most three.

For three vertices \(x,y,z\), write \(\Delta(x,y,z)\) for the
cyclic triangle \(x\to y\to z\to x\), the singleton-block case of
cyclic substitution.
We first record an observation.
Suppose that \(T^\sigma\) is triangle-free, that \(x\) appears
before \(y\) and \(z\), and that \(y\to x\), \(z\to x\), and
\(y\to z\).
Then \(y\) appears before \(z\).
Otherwise \(xy\), \(xz\), and \(yz\) would form a triangle in
\(T^\sigma\).
Consequently, the inneighbors of any vertex that appear after it
induce a transitive tournament and occur in their transitive
order.

Let \(X=\{x_1,x_2,x_3,x_4,x_5\}\) and
\(Y=\{y_1,y_2,y_3,y_4,y_5\}\), and add a vertex \(z\).
Subscripts on the \(y_i\) are read modulo \(5\), with residue
\(0\) written as \(5\).
Start from the transitive tournament with all arcs forward in
the order
\(\sigma_0=(x_1,x_2,x_3,x_4,x_5,y_1,y_2,y_3,y_4,y_5,z)\).
For each \(i\in\{1,2,3,4,5\}\), reverse the arcs on the pairs
\(y_iy_{i+2}\), \(x_iy_{4-i}\), \(x_iy_{5-i}\), and \(x_iz\).
Here the pairs are unordered, so \(\sigma_0\) determines which arc
is reversed in each case.

Let \(C=y_1y_3y_5y_2y_4y_1\).
With respect to \(\sigma_0\), the backward pairs inside \(Y\)
are exactly the edges of \(C\).
Moreover, \(x_i\) is adjacent in \(T^{\sigma_0}\) to the two
neighbors of \(y_{2-i}\) on \(C\), while \(z\) is adjacent to
every vertex of \(X\).
Therefore \(T^{\sigma_0}=M(C)\cong M(C_5)\).
By the Mycielski construction, \(T^{\sigma_0}\) is triangle-free
and \(\chi(T^{\sigma_0})=4\).
Figures~\ref{fig:T} and~\ref{fig:MC5} show the backward arcs and the
corresponding undirected edges.

\begin{figure}[htbp]
\centering
\begin{minipage}[t]{0.48\textwidth}
\centering
\grotzschdrawing{->}
\caption{The backedges of \(T\) in \(\sigma_0\).
Every omitted pair is oriented forward.}
\label{fig:T}
\end{minipage}\hfill
\begin{minipage}[t]{0.48\textwidth}
\centering
\grotzschdrawing{-}
\caption{The graph \(T^{\sigma_0}=M(C)\cong M(C_5)\),
with outer cycle \(C=y_1y_3y_5y_2y_4y_1\).}
\label{fig:MC5}
\end{minipage}
\par\medskip
\begin{minipage}{\textwidth}
\small
Both figures use
\(\sigma_0=(x_1,x_2,x_3,x_4,x_5,y_1,y_2,y_3,y_4,y_5,z)\).
The circular placement displays the five-cycle.
In the Mycielski construction, \(x_i\) is the copy of
\(y_{2-i}\), with subscripts read modulo \(5\).
\end{minipage}
\end{figure}

Let \(\sigma\) be any order for which \(T^\sigma\) is
triangle-free, and let \(v\) be the first vertex of
\(Y\cup\{z\}\) in \(\sigma\).
We show first that all five vertices of \(X\) precede \(v\)
in their reference order, and then force the remaining six
vertices.

We first claim that the vertices of \(X\) appearing before \(v\)
form an initial segment \(x_1,x_2,\ldots,x_k\) in this order,
for some \(0\leq k\leq5\).
Suppose that \(i<j\), but \(x_j\) appears before \(x_i\) and
before \(v\).
The two vertices of \(Y\) that send an arc to \(x_j\) are
\(y_{4-j}\) and \(y_{5-j}\), while the two vertices of \(Y\)
that send an arc to \(x_i\) are \(y_{4-i}\) and \(y_{5-i}\).
Since these two pairs are different, we may choose
\(y_\ell\in
\{y_{4-j},y_{5-j}\}\setminus\{y_{4-i},y_{5-i}\}\).
Then \(x_i\), \(y_\ell\), and \(z\) are later inneighbors of
\(x_j\), and they form the cyclic triangle
\(\Delta(x_i,y_\ell,z)\).
This contradicts the observation above.
The claim follows.

We now show that \(k=5\).
Suppose not.
The vertex \(v\) cannot be \(z\), because \(y_1,y_2,y_3\) are
later inneighbors of \(z\) and form the cyclic triangle
\(\Delta(y_1,y_2,y_3)\).
Hence \(v\in Y\).

Suppose first that \(v=y_1\).
The vertices \(x_5,y_3,y_4\) are inneighbors of \(y_1\) and form
the cyclic triangle \(\Delta(x_5,y_3,y_4)\).
Since \(y_3\) and \(y_4\) appear after \(v\), the observation
forces \(x_5\) to appear before \(v\).
The initial-segment claim then gives \(k=5\), a contradiction.

Next suppose that \(v=y_j\) for \(j\in\{2,3,4\}\). The three vertices
\(x_{6-j},x_{7-j},y_{j-1}\) are inneighbors of \(y_j\) and, in this
order, form a cyclic triangle. Since \(y_{j-1}\) follows \(v\),
at least one of these two \(x\)-vertices precedes \(v\). The
initial-segment claim then puts \(x_{5-j}\) before \(v\).
Its later inneighbors include \(y_{j-1}\) and \(y_j\), with
\(y_{j-1}\to y_j\). They must therefore occur in that order,
contrary to the choice of \(v\).

Finally, suppose that \(v=y_5\).
The vertices \(x_3,y_4,y_1\) are inneighbors of \(y_5\) and form
the cyclic triangle \(\Delta(x_3,y_4,y_1)\).
Thus \(x_3\) appears before \(v\), so \(x_1,x_2,x_3\) all appear
before every vertex of \(Y\).
Applying the observation at \(x_3,x_2,x_1\), respectively, gives
\(y_1\prec_\sigma y_2\prec_\sigma y_3\prec_\sigma y_4\).
But \(y_1\) and \(y_4\) are later inneighbors of \(y_5\), and
\(y_4\to y_1\).
The observation therefore gives \(y_4\prec_\sigma y_1\), a
contradiction.

It follows that \(k=5\), so \(\sigma\) begins with
\(x_1,x_2,x_3,x_4,x_5\) in that order. Applying the observation at
\(x_3,x_2,x_1,x_5,x_4\), respectively, forces
\[
y_1\prec_\sigma y_2\prec_\sigma y_3\prec_\sigma y_4
\prec_\sigma y_5\prec_\sigma z .
\]
Thus \(\sigma_0\) is the unique order with a triangle-free backedge graph.

It remains to determine \(\omega(T)\) and \(\chi(T)\).
Since \(T^{\sigma_0}\) is triangle-free, we have \(\omega(T)\leq2\).
On the other hand, \(T\) is not transitive, since
\(\Delta(x_1,y_1,z)\) is a cyclic triangle.
Therefore \(\omega(T)=2\).

The following three lists partition \(V(T)\), each in transitive order:
\[
(x_1,x_2,x_3,x_4,x_5),\qquad (y_1,y_5,y_2),\qquad (y_3,y_4,z).
\]
Hence \(\chi(T)\leq3\).

Suppose that \(\chi(T)\leq2\).
Then some order \(\sigma\) has a bipartite backedge graph.
Since every bipartite graph is triangle-free, the uniqueness
proved above gives \(\sigma=\sigma_0\).
But \(\chi(T^{\sigma_0})=4\), a contradiction.
Therefore \(\chi(T)=3\).

Any \(\omega\)-order of \(T\) has a triangle-free backedge graph,
because \(\omega(T)=2\). Hence it must be \(\sigma_0\).
But \(\chi(T^{\sigma_0})=4>3=\chi(T)\), so \(\sigma_0\) is not
a \(\chi\)-order.
Thus \(T\) has no simultaneously optimal order, and
\(\chi_\omega(T)=4\).

We finally show that eleven vertices are necessary.
Let \(T_0\) be a tournament on at most ten vertices.
By Neumann-Lara's result, \(\chi(T_0)\leq3\).
If \(\omega(T_0)=\chi(T_0)\), then any \(\chi\)-order is
simultaneously optimal.
Suppose instead that \(\omega(T_0)<\chi(T_0)\).
Since \(\omega(T_0)=1\) only when \(T_0\) is transitive, we must
have \(\omega(T_0)=2\) and \(\chi(T_0)=3\).

Take an \(\omega\)-order \(\sigma\) of \(T_0\).
Then \(T_0^\sigma\) is triangle-free.
Since it has at most ten vertices, Chv\'atal's result gives
\(\chi(T_0^\sigma)\leq3\).
Since \(\chi(T_0)=3\), we have \(\chi(T_0^\sigma)=3\).
Thus \(\sigma\) is both an \(\omega\)-order and a \(\chi\)-order.

Therefore eleven is the minimum order of a tournament with no
simultaneously optimal order.
\end{proof}

\section{Clique number under substitution}
\label{sec:substitution}

For a graph \(G\) with positive vertex weights \(w\), define
\[
\omega_w(G)=\max\left\{\sum_{v\in C}w(v):
C\text{ is a clique of }G\right\}.
\]
For a tournament \(T\) with positive vertex weights \(w\), put
\(\Omega(T;w)=\min_\pi\omega_w(T^\pi)\).

\begin{lemma}
\label{lem:substitution-bounds}
Let \((T,\mathcal S)\) be a substitution and put
\(w(v)=\omega(S_v)\).
For every nonempty \(X\subseteq V(T)\),
\[
\omega((T,\mathcal S))
\geq\omega(T[X])+\min_{v\in X}w(v)-1.
\]
The minimum backedge clique number over orders in which every
block is consecutive equals \(\Omega(T;w)\). Consequently,
\begin{equation}
\label{eq:substitution-product}
\omega((T,\mathcal S))\leq\Omega(T;w)
\leq\omega(T)\max_{v\in V(T)}w(v).
\end{equation}
\end{lemma}

\begin{proof}
Fix a nonempty set \(X\subseteq V(T)\) and take any linear order
\(\sigma\) of \((T,\mathcal S)\).
For each \(v\in X\), choose a clique \(K_v\) of size
\(w(v)\) in the backedge graph of \(S_v\) under the restricted order, and let \(r_v\) be its first vertex.
The representatives \(r_v\) with \(v\in X\) induce a copy of
\(T[X]\), so some \(\omega(T[X])\) of them form a backedge
clique \(C\).
Let \(r_z\) be the last vertex of \(C\).

Every vertex of \(K_z\setminus\{r_z\}\) comes after \(r_z\).
Moreover, if \(r_v\in C\setminus\{r_z\}\), then \(r_vr_z\)
is a backedge, so \(z\to v\) in \(T\).
Hence every vertex of \(K_z\) points to \(r_v\).
It follows that \(K_z\cup(C\setminus\{r_z\})\) is a backedge
clique of size \(w(z)+\omega(T[X])-1\), which is at least
\(\min_{v\in X}w(v)+\omega(T[X])-1\).
Since \(\sigma\) and \(X\) were arbitrary, the result follows.

Fix an order \(\pi\) of \(T\) and an internal order \(\tau_v\)
of each block. Put \(k(v)=\omega(S_v^{\tau_v})\), and concatenate
the internal orders according to \(\pi\).
The blocks met by a backedge clique form a clique of \(T^\pi\).
Conversely, maximum internal backedge cliques can be united over
any clique of \(T^\pi\).
The resulting backedge clique number is therefore
\[
\max\left\{\sum_{v\in C}k(v):C\text{ is a clique of }T^\pi\right\}.
\]
This expression is nondecreasing in each \(k(v)\).
Moreover, \(k(v)\geq w(v)\), with equality simultaneously
attainable by choosing an \(\omega\)-order of each block.
Minimizing over the internal orders and then over \(\pi\)
proves the result.

Finally, choose an \(\omega\)-order of \(T\).
Every clique of its backedge graph has at most \(\omega(T)\)
vertices and hence total weight at most
\(\omega(T)\max_v w(v)\), proving the second inequality
in \eqref{eq:substitution-product}.
\end{proof}

For transitive \(T\), placing the blocks in its transitive order
gives \(\omega((T,\mathcal S))=\max_v\omega(S_v)\);
the reverse inequality follows by restriction.
For \(C_3\) with positive weights \(w_1,w_2,w_3\),
\begin{equation}
\label{eq:weighted-triangle}
\Omega(C_3;w)
=\max\left\{w_1,w_2,w_3,\min_{1\leq i<j\leq3}(w_i+w_j)\right\}.
\end{equation}
Indeed, every order has a backedge, and any chosen pair can be
the unique backedge in a suitable order.

The outer tournament and the block clique numbers alone do not
determine the clique number of a substitution.

\begin{proposition}
\label{prop:scalar-counterexample}
Let \(H=\Delta(C_3,C_3,K_1)\) and \(K=\Delta(H,C_3,K_1)\).
Then \(\omega(H)=2\) and \(\omega(K)=3\).
Both substitutions have outer tournament \(C_3\) and block
clique numbers \((2,2,1)\).
\end{proposition}

\begin{proof}
Label the triangle blocks of \(H\) by \(a_1,a_2,a_3\) and
\(b_1,b_2,b_3\), in cyclic order, and its singleton by \(c\).
In the order
\[
\sigma=(a_1,a_2,b_1,b_2,c,a_3,b_3),
\]
the backedge graph is the five-cycle
\(a_1a_3b_1b_3ca_1\) with pendant edges \(a_2c\) and \(b_2a_3\).
It is triangle-free, so \(\omega(H)\leq2\).
Since \(H\) contains a cyclic triangle, \(\omega(H)=2\).

We also have \(\chi(H)=3\).
Indeed, in a two-coloring each triangle block would meet both
color classes. The class containing \(c\) would then meet all
three blocks and contain a cyclic triangle.
On the other hand, the three sets
\[
\{c\},\qquad \{a_1,a_2,b_1,b_2\},\qquad \{a_3,b_3\}
\]
are transitive and partition \(V(H)\).

The tournament \(K\) is isomorphic to \(\widehat H\), after a
cyclic relabeling of its three blocks.
Since \(\chi(H)=3\), Lemma~\ref{lem:four-vertices} gives
\(\omega(K)\geq3\).
Lemma~\ref{lem:substitution-bounds} and
\eqref{eq:weighted-triangle} give
\(\omega(K)\leq\Omega(C_3;(2,2,1))=3\).
Thus \(\omega(K)=3\).
\end{proof}

The two substitutions have the same outer tournament and block
clique numbers but different clique numbers. Also,
\(\omega(H)=2<3=\Omega(C_3;(2,2,1))\), so an optimal order
need not keep the blocks consecutive.

Aubian and Coulomb~\cite{AubianCoulomb2026} constructed infinitely
many three-clique-critical circulant tournaments \(T\) and proved
that \(\Delta(T,T,T)\) is four-clique-critical.
We obtain bounds for unequal blocks and an exact criticality
criterion when the blocks satisfy \(\omega=\chi\).
Put \(\Phi(x,y,z)=\max\{x,y,z,1+\min\{x,y,z\}\}\).

\begin{proposition}
\label{prop:cyclic-substitution}
Let \(A,B,C\) be pairwise disjoint nonempty tournaments and put
\(R=\Delta(A,B,C)\).
\begin{enumerate}[label=\textup{(\roman*)}]
\item We have
\(\Phi(\omega(A),\omega(B),\omega(C))
\leq\omega(R)\leq
\Phi(\chi(A),\chi(B),\chi(C))\).
If every block \(X\in\{A,B,C\}\) satisfies \(\omega(X)=\chi(X)\),
then
\(\omega(R)=\Phi(\omega(A),\omega(B),\omega(C))\).

\item Let \(a,b,c\geq2\). Suppose that \(A,B,C\) are, respectively,
\(a\)-, \(b\)-, and \(c\)-\(\chi\)-critical, and that
\(\omega(A)=a\), \(\omega(B)=b\), and \(\omega(C)=c\).
Then \(R\) is \(\omega\)-critical if and only if \(a=b=c\).
If \(a=b=c=k\), then \(R\) is \((k+1)\)-\(\omega\)-critical.
\end{enumerate}
\end{proposition}

\begin{proof}
For (i), apply Lemma~\ref{lem:substitution-bounds} with
outer tournament \(C_3\).
A singleton \(X\) contributes the clique number of its block,
a two-vertex set contributes the smaller of the two block
parameters, and \(X=V(C_3)\) contributes
\(1+\min\{\omega(A),\omega(B),\omega(C)\}\).
This proves the lower bound.

For the upper bound, put \(a=\chi(A)\), \(b=\chi(B)\), and
\(c=\chi(C)\).
Relabel the blocks cyclically if needed, so that \(a\leq b\)
and \(a\leq c\).
Partition the three blocks into transitive classes
\(A_1,\ldots,A_a\), \(B_1,\ldots,B_b\), and \(C_1,\ldots,C_c\).
Order the classes as
\(B_1,C_1,A_1,B_2,C_2,A_2,\ldots\), omitting nonexistent classes,
and use a transitive order within each class.
A backedge clique contains at most one vertex from each class.

A vertex of \(A_i\) and a vertex of \(B_j\) form a backedge
exactly when \(j\leq i\).
Thus, if a backedge clique meets both \(A\) and \(B\) and \(t\)
is the least index of an \(A\)-class that it meets, then it meets
at most \(t\) classes of \(B\) and at most \(a-t+1\) classes
of \(A\). Its size is therefore at most \(a+1\).
A clique contained in \(A\) has size at most \(a\), and one
contained in \(B\) has size at most \(b\).
Therefore, every clique in \(A\cup B\) has size at most
\(\max\{a+1,b\}\).

Similarly, a vertex of \(B_i\) and a vertex of \(C_j\) form a
backedge exactly when \(j<i\).
If \(t\) is the least index of a \(B\)-class met by a clique
meeting both blocks, its size is at most
\((b-t+1)+(t-1)=b\).
Thus every clique in \(B\cup C\) has size at most \(\max\{b,c\}\).
For \(C_i\) and \(A_j\), the backedge condition is again \(j<i\).
With \(t\) the least \(C\)-index in a clique meeting both blocks,
at most \(c-t+1\) classes of \(C\) and \(t-1\) classes of \(A\)
are available. A clique contained in one of these blocks also has
size at most \(c\), since \(a\leq c\).

Finally, no backedge clique can meet all three blocks:
choosing one vertex from each would give a cyclic triangle all
of whose arcs point backwards in a linear order, which is
impossible.
Therefore, the constructed order has backedge clique number at
most \(\max\{a+1,b,c\}=\Phi(a,b,c)\), where the equality uses
\(a=\min\{a,b,c\}\).
This proves the upper bound.
If all three blocks satisfy \(\omega=\chi\), the two bounds
coincide.

For (ii), suppose that \(A,B,C\) satisfy the stated criticality
assumptions.
Let \(v\) belong to a block \(X\) with \(\omega(X)=\chi(X)=q\).
Since deleting one vertex decreases the clique number by at most
one and \(\omega\leq\chi\), we have
\(q-1=\omega(X)-1\leq\omega(X-v)\leq\chi(X-v)=q-1\).
Thus deleting a vertex reduces both parameters of its block from
\(q\) to \(q-1\), and part~(i) computes the clique number before
and after every vertex deletion.

If \(a=b=c=k\), then \(\Phi(k,k,k)=k+1\) and
\(\Phi(k-1,k,k)=k\), with the same calculation for the other two
blocks.
Thus \(\Delta(A,B,C)\) is \((k+1)\)-\(\omega\)-critical.

Now suppose that \(a,b,c\) are not all equal, and put
\(m=\min\{a,b,c\}\) and \(M=\max\{a,b,c\}\).
Then \(M\geq m+1\), so \(\Phi(a,b,c)=M\).
Deleting a vertex from a block of clique number \(m\) does not
change the maximum value \(M\), and hence does not change the
value of \(\Phi\).
Therefore \(\Delta(A,B,C)\) is not \(\omega\)-critical.
\end{proof}

The two bounds can be arbitrarily far apart.
Tournaments with clique number two and unbounded dichromatic
number follow from an observation of Nguyen, Scott, and
Seymour~\cite{NguyenScottSeymour2025}; Theorem~\ref{thm:poset}
gives such examples among poset tournaments.
For three copies of such a tournament \(T\), the bounding
expressions are \(3\) and \(\chi(T)+1\).

\section*{Formal verification}

A Lean~4 formalization of the main results is available in
release~v0.1.0 of the companion repository~\cite{formalization}.
It includes the critical Cayley families, the four-vertex reduction,
the poset separation, the eleven-vertex obstruction and its minimality,
and the substitution results. The complexity theorems take the cited
hardness results as hypotheses in a uniform matrix model. For
Theorem~\ref{thm:poset}, the formalization uses an explicit recursive
family. The additional assertion about prescribed girth is proved
here using Trotter and Wang's theorem.

\section*{Acknowledgments}

The authors are grateful to Pierre Aboulker and Samuel Coulomb
for their helpful correspondence and generosity in sharing
information about these results.


\begin{thebibliography}{99}

\bibitem{AboulkerEtAl2026}
P.~Aboulker, G.~Aubian, P.~Charbit, and R.~Lopes,
\emph{Clique number of tournaments},
Electron. J. Combin. \textbf{33} (2026), no.~3, Paper~P3.56.
\doi{10.37236/12557}.

\bibitem{AboulkerEtAlPOD2026}
P.~Aboulker, L.~Crew, J.~Duron, X.~Fan, H.~Jacob,
R.~Kimbrough, H.~Koerts, B.~Moore, S.~Spirkl, and S.~Thomass\'e,
\emph{Decomposing tournaments into comparability graphs},
preprint, 2026, \arxiv{2606.07748v1}.

\bibitem{araujo-olsen}
G.~Araujo-Pardo and M.~Olsen,
\emph{A conjecture of Neumann-Lara on infinite families of
\(r\)-dichromatic circulant tournaments},
Discrete Mathematics \textbf{310} (2010), 489--492.

\bibitem{Aubian2024}
G.~Aubian,
\emph{Computing the clique number of tournaments},
preprint, 2024, \arxiv{2401.07776}.

\bibitem{AubianCoulomb2026}
G.~Aubian and S.~Coulomb,
\emph{Clique Number of Tournaments II},
preprint, 2026, \arxiv{2609.07481v1}.

\bibitem{ChenWangCayley2026}
G.~Chen and S.~Wang,
\emph{Cayley Tournaments Simultaneously Critical for the Clique and
Dichromatic Numbers}, preprint, 2026, \arxiv{2609.08658v1}.

\bibitem{Chvatal1974}
V.~Chv\'atal,
\emph{The minimality of the Mycielski graph},
in R.~A.~Bari and F.~Harary (eds.),
\emph{Graphs and Combinatorics},
Lecture Notes in Mathematics \textbf{406},
Springer, Berlin, 1974, 243--246.
\doi{10.1007/BFb0066446}.

\bibitem{CrewEtAl2026}
L.~Crew, X.~Fan, H.~Koerts, B.~Moore, and S.~Spirkl,
\emph{Characterizing Large Clique Number in Tournaments},
preprint, 2026, \arxiv{2602.09863v1}.

\bibitem{FoxGishbolinerShapiraYuster2019}
J.~Fox, L.~Gishboliner, A.~Shapira, and R.~Yuster,
\emph{The removal lemma for tournaments},
J. Combin. Theory Ser. B \textbf{136} (2019), 110--134.
\doi{10.1016/j.jctb.2018.10.001}.

\bibitem{Kim2013}
I.~Kim,
\emph{On Containment Relations in Directed Graphs},
Ph.D. thesis, Princeton University, 2013.
\href{https://dataspace.princeton.edu/handle/88435/dsp012801pg46x}
{Princeton DataSpace: 88435/dsp012801pg46x}.

\bibitem{lelarge-k5}
M.~Lelarge,
\emph{An infinite family of \(5\)-\(\bar{\omega}\)-critical
tournaments},
unpublished manuscript,
\url{https://github.com/LLM4Rocq/digraph-theory/blob/main/paper/k5_theorem.pdf}.

\bibitem{Mycielski1955}
J.~Mycielski,
\emph{Sur le coloriage des graphes},
Colloq. Math. \textbf{3} (1955), 161--162.

\bibitem{NeumannLara1982}
V.~Neumann-Lara,
\emph{The dichromatic number of a digraph},
J. Combin. Theory Ser. B \textbf{33} (1982), no.~3, 265--270.
\doi{10.1016/0095-8956(82)90046-6}.

\bibitem{NeumannLara1994}
V.~Neumann-Lara,
\emph{The 3 and 4-dichromatic tournaments of minimum order},
Discrete Math. \textbf{135} (1994), 233--243.
\doi{10.1016/0012-365X(93)E0113-I}.

\bibitem{neumann1997}
V.~Neumann-Lara,
\emph{Vertex critical 4-dichromatic circulant tournaments},
Discrete Mathematics \textbf{170} (1997), 289--291.

\bibitem{neumann2000}
V.~Neumann-Lara,
\emph{Dichromatic number, circulant tournaments and Zykov sums
of digraphs},
Discussiones Mathematicae Graph Theory \textbf{20} (2000),
197--207.

\bibitem{neumann-urrutia}
V.~Neumann-Lara and J.~Urrutia,
\emph{Vertex critical \(r\)-dichromatic tournaments},
Discrete Mathematics \textbf{49} (1984), 83--87.

\bibitem{NguyenScottSeymour2025}
T.~Nguyen, A.~Scott, and P.~Seymour,
\emph{Some results and problems on tournament structure},
J. Combin. Theory Ser. B \textbf{173} (2025), 146--183.
\doi{10.1016/j.jctb.2025.02.002}.

\bibitem{formalization}
\emph{Tournament Clique Number},
Lean~4 formalization, release~v0.1.0, December 2026.
\url{https://github.com/Swang85/tournament-clique-number/releases/tag/v0.1.0}.

\bibitem{TrotterWang2014}
W.~T.~Trotter and R.~Wang,
\emph{Incidence posets and cover graphs},
Order \textbf{31} (2014), 279--287.
\doi{10.1007/s11083-013-9301-9}.

\end{thebibliography}
\end{document}